\documentclass[11pt, reqno]{amsart}
\usepackage{amssymb}
\usepackage{latexsym}
\usepackage{amsmath}
\usepackage{amsfonts}
\usepackage[hidelinks]{hyperref}
\hypersetup{
  colorlinks   = true, 
  urlcolor     = blue, 
  linkcolor    = blue, 
  citecolor   = blue 
}
\usepackage{enumitem}
\usepackage{fancyhdr}
\usepackage{cleveref}
\usepackage{mathrsfs}

\usepackage{color}
\theoremstyle{plain}
\newtheorem{theorem}{Theorem}[section]

\newcommand{\R}{\ensuremath{\mathbb{R_+}}}

\newtheorem{lemma}[theorem]{Lemma}

\newtheoremstyle{remark}
    {} 
    {} 
    {}          
    {}          
    {\bfseries} 
    {.}         
    {.5em}      
    {}          

\theoremstyle{remark}
\newtheorem{remark}{Remark}[section]

\newtheoremstyle{example}
    {\dimexpr\topsep/2\relax} 
    {\dimexpr\topsep/2\relax} 
    {}          
    {}          
    {\bfseries} 
    {.}         
    {.5em}      
    {}          

\theoremstyle{example}

\newtheoremstyle{definition}
    {\dimexpr\topsep/2\relax} 
    {\dimexpr\topsep/2\relax} 
    {}          
    {}          
    {\bfseries} 
    {.}         
    {.5em}      
    {}          

\theoremstyle{definition}
\newtheorem{definition}{Definition}[section]

\newtheorem*{similartheorem*}{Theorem \dualnumber{$'$}}

\numberwithin{equation}{section}

\def\R{\mathbb{R}}
\def\H{\mathbb{H}}
\def\C{\mathbb{C}}
\def\Z{\mathbb{Z}}

\allowdisplaybreaks

\begin{document}
\title[Bilinear fractional integral operator on the Heisenberg group]{Weighted estimates for bilinear fractional integral operator on the Heisenberg group}

{\let\thefootnote\relax}

\thanks{The first author is supported by Tata Institute of Fundamental Research, Centre for Applicable Mathematics, Bangalore.}
\thanks{The second author is supported by NBHM, Government of India.}

\author{Abhishek Ghosh and Rajesh K. Singh}
\address[Abhishek Ghosh]{Tata Institute of Fundamental Research, Centre for Applicable Mathematics, Bangalore--560065, Karnataka, India.}
\email{abhi170791@gmail.com, abhi21@tifrbng.res.in}

\address[Rajesh K. Singh]{Department of Mathematics, Indian Institute of Science, Bangalore-560012, Karnataka, India.}
\email{agsinghraj@gmail.com, rajeshsingh@iisc.ac.in}

\pagestyle{headings}

\begin{abstract}
In this article, we introduce an analogue of Kenig and Stein's bilinear fractional integral operator on the Heisenberg group $\mathbb{H}^n$. We completely characterize exponents $\alpha, \beta$ and $\gamma$ such that the operator is bounded from $L^{p}(\mathbb{H}^n, |x|^{\alpha p})\times L^{q}(\mathbb{H}^n, |x|^{\beta q})$ to $L^{r}(\mathbb{H}^n, |x|^{-\gamma r})$.
\end{abstract}

\maketitle

\section{Introduction and preliminaries}\label{Introduction}
Fractional integral operators are classical objects in analysis pertaining to the study of smoothness of functions, potential theory and embedding theorems. Recall that for $0<\lambda<n$, the fractional integral operator is defined as follows
\begin{equation*}
    I_{\lambda}f(x)=\int_{\mathbb{R}^n}\frac{f(y)}{|x-y|^{n-\lambda}}\ dy, \ \ \ x \in \R^{n}.
\end{equation*}
The operators $I_{\lambda}$ are bounded off-diagonally and the characterization of weights for which $I_{\lambda}: L^p(w^{p})\mapsto L^q(w^q),$ with ${1}/{q}={1}/{p}-{\lambda}/{n},\,\,1<p<n/\lambda,$  was obtained by Muckenhoupt and Wheeden in \cite{MW-fractional}. The appropriate class of weights are denoted as $A_{p, q}$ weights.
 The operator $I_\lambda$ and its analogues are also investigated beyond the Euclidean setting. 
 
 In this article we are interested in  bilinear analogue of $I_{\lambda}$ on the Heisenberg group $\mathbb{H}^n$. Let us begin with the bilinear fractional integral operator $BI_{\lambda}$ on $\mathbb{R}^n$  defined as 
\begin{align*}
BI_{\lambda}(f, g)(x)=\int_{\mathbb{R}^n}\frac{f(x-y)g(x+y)}{|y|^{n-\lambda}}\, dy, \, \,\, 0<\lambda<n.   
\end{align*}
These operators are well studied, for example we refer the works \cite{Grafakos-1992, Kenig-Stein-MRL, Kabe-TAMS}. They are also of interest due to their connections with the bilinear Hilbert transform of Lacey and Thiele (see \cite{Lacey-Thiele}). It was proved in \cite{Kenig-Stein-MRL} that $BI_{\lambda}$ bounded from  $L^{p}(\mathbb{R}^n)\times L^{q}(\mathbb{R}^n)$ to $L^{r}(\mathbb{R}^n)$ provided  $1/r=1/p+1/q-\lambda/n>0$ and $1<p, q\leq \infty$,  and also the expected weak type inequality holds if either $p$ or $q$ is $1$. It is not difficult to see that using H\"older's inequality and weighted boundedness of $I_{\lambda}$, we can obtain that $BI_{\lambda}: L^{p}(w_1^{p})\times L^{q}(w_1^{q})\mapsto L^{r}(w_{1}^r w_{2}^{r})$ provided $1/r=1/p+1/q-\lambda/n,$ $1<r, s<\infty$ and $w_{i}^{p/s}\in A_{p, q}$, where $1/s=1/p+1/q$. However, the above approach is not useful when $r<1$ and it was also pointed out in the influential work of Lerner \textit{et al} that linear Muckenhoupt classes are not the appropriate weights while studying bilinear operators. In \cite{Lerner-Adv-Multi-Weight}, multilinear $\mathcal{A}_{\vec{P}}$ weights are introduced in connection with the multilinear Hardy--Littlewood maximal operator and multilinear Calder\'on--Zygmund operators. Subsequently, Kabe Moen has initiated the study of fractional multilinear weights and proved the following: $BI_{\lambda}$ maps $L^{p}(w_1^{p})\times L^{q}(w_2^{q})$ to $L^r(w_1^r w_2^r)$ boundedly, when $1<p, q<\infty$, ${1}/{r}={1}/{p}+1/q-\lambda/n>1,$ and $\vec{w}\in \mathcal{A}_{p, q, r}$, where the the class $ \mathcal{A}_{p, q, r}$ is defined as follows.  We say $\vec{w}=(w_{1}, w_{2})\in \mathcal{A}_{p, q, r}$ if
\begin{align*}
\sup_{Q}\left(\frac{1}{|Q|}\int_{Q}(w_1 w_{2})^{r/(1-r)}\right)^{(1-r)/r}\left(\frac{1}{|Q|}\int_{Q} w_{1}^{-p'}\right)^{1/p'}\left(\frac{1}{|Q|}\int_{Q} w_{2}^{-q'}\right)^{1/q'}\lesssim C<\infty,
\end{align*}
where the supremum is over all cubes with sides parallel to the coordinate axes. Though it is not yet known whether the condition $\mathcal{A}_{p, q, r}$ is also necessary for the boundedness of $BI_{\alpha}$. Interestingly, if we only consider power weights then it was shown in \cite{Komori-2020} that it is possible to obtain both necessary and sufficient conditions on $\alpha, \beta,$ and $\gamma$ such that $BI_{\lambda}$ is bounded from  $L^{p}(|x|^{\alpha p})\times L^{q}(|x|^{\beta q})$ to $L^{r}(|x|^{-\gamma r})$, in the particular case when $\gamma$ satisfies $\gamma=-\alpha-\beta$. Our primary goal in this article is to obtain a complete characterization of $\alpha, \beta,$ and $\gamma$ in full generality such that the bilinear fractional operator $B_{\lambda}$ maps $L^{p}(|x|^{\alpha p})\times L^{q}(|x|^{\beta q})$ to $L^{r}(|x|^{-\gamma r})$ on the Heisenberg group $\mathbb{H}^n$. We write our results on the Heisenberg group but it is not difficult to see that the same ideas also work for the Euclidean case, which in turn improves all the existing results in the Euclidean setting. We will explain it in detail in Remark~\ref{comparison}.

To illustrate our results, let us recall the following preliminaries. Let $\mathbb{H}^n:= \mathbb{C}^n \times \mathbb{R}$ denotes the $(2n+1)$-dimensional Heisenberg group with the group law
\begin{equation}
(z,t) \cdot (w,s) = \left(  z + w, t + s + \frac{1}{2}  \Im(  z \cdot \bar{w}) \right),\ \ \text{for all} \ \ (z,t), (w,s) \in \mathbb{H}^n.
\end{equation} 
We have a family of non-isotropic dilations defined by $\delta_{r}(z,t):=(rz,r^2t)$, for all $(z,t) \in \mathbb{H}^n$, for every $r>0$. The Koranyi norm on $\mathbb{H}^n$ is defined by
 \begin{equation*}
     |(z,t)|:= \left( |z|^4 + t^2 \right)^{\frac{1}{4}}, \ \ \    (z,t) \in \mathbb{H}^n, 
 \end{equation*}
which is homogeneous of degree 1, that is $|\delta_{r}(z,t)|= r \,  |(z,t)|$. The Haar measure on $\mathbb{H}^n$ coincides with the Lebesgue measure $dz dt$. Let $B(0,r)=\{ (z,t) \in \H^n : |(z,t)| < r \}$ be the ball of radius $r$ with respect to Koranyi norm. One has its measure $|B(0,r)| = C_{Q} \, r^{Q}$, where $Q=(2n + 2)$ is known as the homogeneous dimension of $\H^n$. The convolution of $f$ with $g$ on $\H^{n}$ is defined by
\begin{equation*}
    f * g \, (x) = \int_{\H^{n}}  f(x y^{-1}) g(y) dy, \ \ \ x \in \H^{n}.
\end{equation*}

Recall that the fractional integral operator, $\mathcal{I}_{\lambda}$, on the Heisenberg group $\mathbb{H}^n$ is defined as follows $$\mathcal{I}_{\lambda}f(x)=\int_{\mathbb{H}^n}f(xy^{-1})\,~~\frac{dy}{|y|^{Q-\lambda}}, \ \ \ 0<\lambda<Q.$$ Fractional integral operators on the Heisenberg group has a long history, starting with the foundational work of Folland and Stein in \cite{Folland-Stein}, where it was shown that $\mathcal{I}_{\lambda}: L^p(\mathbb{H}^n)\mapsto L^{q}(\mathbb{H}^n)$ with $1/q=1/p-\lambda/Q, 1<p<Q/\lambda$, also the natural end-point boundedness $\mathcal{I}_{\lambda}: L^1(\mathbb{H}^n)\mapsto L^{Q/(Q-\lambda), \infty}(\mathbb{H}^n)$. We would like to mention the work \cite{Kairema-homogeneous} where fractional integral operators are extended in the more general context of spaces of homogeneous type, and also the article \cite{Fractional-Lie-Group} where the authors have treated an analogue of Kenig and Stein's bilinear fractional integral operator on compact Lie groups. Motivated by the above discussion, let us define bilinear fractional integral operator on $\mathbb{H}^n$.
\begin{definition}
For $0<\lambda<Q$, the bilinear fractional integral operator $B_{\lambda}$ on $\mathbb{H}^n$ is defined as follows
\begin{equation*}
B_{\lambda}(f,g)(x)= \int_{\mathbb{H}^n}  f(x y^{-1}) \, g(xy)  \, \frac{dy}{|y|^{Q-\lambda}}.
\end{equation*}
\end{definition}

Our first result addresses the unweighted boundedness of the operator $B_{\lambda}$ on $\mathbb{H}^n$, thus extending the result of Kenig and Stein to the Heisenberg group.
\begin{theorem}[Unweighted boundedness]\label{Kenig Stein Hn}
Let $0<\lambda<Q$, $\frac{1}{r}= \frac{1}{p} + \frac{1}{q}- \frac{\lambda}{Q} >0$, and that $f \in L^p(\H^n)$, $g \in L^q(\H^n)$, $1 \leq p,q \leq \infty$. Then,
\begin{enumerate}
\item [\emph{(\emph{a})}]  If $1<p,q \leq \infty$,
\begin{equation}
\label{B alpha bddness unwtd}
\| \,  B_{\lambda}(f,g) \|_{L^{r}(\mathbb{H}^n  )} \leq K  \, \|  f \|_{L^p( \mathbb{H}^n  )} \| g \|_{L^q( \mathbb{H}^n  )},
\end{equation}
\item [\emph{(\emph{b})}] If $1 \leq p,q \leq \infty$, with either $p=1$ or $q=1$,
\begin{equation}
\label{B alpha weak bddness unwtd}
\| \,  B_{\lambda}(f,g) \|_{L^{r, \infty}(\mathbb{H}^n  )} \leq K  \, \|  f \|_{L^p( \mathbb{H}^n  )} \| g \|_{L^q( \mathbb{H}^n  )}.
\end{equation}
\end{enumerate}
\end{theorem}

Now, we present the main result of this article concerning the characterization of power weights for the boundedness of $B_{\lambda}$. Precisely, we obtain the following:

\begin{theorem}[Characterization of power weights]\label{Power weights C}
Let $1<p,q<\infty$, $0<r<\infty$ and $\frac{1}{r} \leq \frac{1}{p} + \frac{1}{q}$. Let  $0< \lambda <Q$  and  
\begin{equation}\label{Easy ABC}
\alpha < \frac{Q}{p'}, \ \ \ \beta < \frac{Q}{q'} \ \ \  \text{and}   \ \ \ \gamma < \frac{Q}{r}.
\end{equation}
Further, assume $\frac{1}{r}= \frac{1}{p} + \frac{1}{q}- \frac{\lambda - \alpha - \beta -\gamma}{Q} >0$.

Then, the following are equivalent:
\begin{enumerate}
\setlength\itemsep{.5em}
\item [\emph{(\emph{a})}] There exists a constant $K>0$, such that
\begin{equation}
\label{suff-cond-ABC}
\| \, |x|^{-\gamma} B_{\lambda}(f,g) \|_{L^{r}(\mathbb{H}^n  )} \leq K  \| |x|^{\alpha} f \|_{L^p( \mathbb{H}^n  )} \| |x|^{\beta}g \|_{L^q( \mathbb{H}^n  )},
\end{equation}
for all $f \in L^{p}(\mathbb{H}^n)$ and $g \in L^{q}(\mathbb{H}^n)$;
\item [\emph{(\emph{b})}] The exponents $\alpha, \beta$ and $\gamma$ satisfy

\begin{equation}\label{Necessary-condn-ABC}
\textnormal{(I)} \  -Q + \lambda \leq  \beta + \gamma, \ \ \  \textnormal{(II)} \  -Q + \lambda \leq \gamma +  \alpha, \ \ \ \textnormal{(III)} \  -Q + \lambda \leq \alpha + \beta,  \  \  and \ \ \textnormal{(IV)} \  \alpha+\beta+\gamma \geq 0.
\end{equation}

\end{enumerate}
\end{theorem}
\medskip
The proof of the above theorem is quite involved. Since the operator $B_\lambda$ involves the product of the form $f(xy^{-1}) g(xy)$, in order to obtain the necessary conditions we need to construct functions $f$ and $g$ with very delicate precision. In the following remark we would like to mention some key features of Theorem~\ref{Power weights C}.
\begin{remark}
\label{comparison}
We point out that our characterization, that is, Theorem~\ref{Power weights C} is sharp. Moreover, if restricted to the Euclidean setting, it improves all the previously known results. We mention some of them here.
\begin{itemize}
\item
Recall that the condition obtained in Theorem~2 in \cite{Komori-2020} were $\alpha\leq n-\lambda, \beta\leq n-\lambda,$ and $-n+\lambda\leq \alpha+\beta$ which certainly implies our conditions. Also it is easy to see that by suitably choosing $\epsilon>0$ we can construct triplets $\alpha=n-\lambda+\epsilon, \beta=0, \gamma=-n+\lambda$ which satisfy conditions of Theorem~\ref{Power weights C} but it is not covered by Theorem~2 in \cite{Komori-2020}.
\item
A simple computation shows that our result also recovers the well-known result of Hoang and Moen, see Theorem~10.1 in \cite{Moen-Hoang}. 
\item
In \cite{Brascamp-Lieb-PAMS}, as an application of Brascamp-Lieb forms, the authors have obtained boundedness  \eqref{suff-cond-ABC} when one has (\ref{Necessary-condn-ABC}) with  strict inequalities in the first three conditions and the $r$ is additionally restricted by the requirement $1<r<\infty$. The limitation $1<r<\infty$ is due to use of duality arguments in their proof.
\end{itemize}
\end{remark}

Another interesting aspect related to the study of fractional integral operator is the  Stein-Weiss inequality obtained in 1958 by Stein and Weiss in \cite{Stein-Weiss}. We recall it here. The inequality
\begin{align}
\label{Stein-Weiss-Linear}
\Bigg|\int_{\mathbb{R}^n}\int_{\mathbb{R}^n}\frac{\overline{f(x)}g(y)}{|x|^{\alpha}|x-y|^{\lambda}|y|^{\beta}}\, dx\, dy\Bigg|\lesssim \|f\|_{L^p(\mathbb{R}^n)}\|g\|_{L^q(\mathbb{R}^n)}, 
\end{align}
holds, where $1<p, q<\infty$, $0<\lambda<n$, $\alpha +\beta\geq 0$,  $1/p+1/q+(\alpha+\beta+\lambda)/n=2$ with  $\alpha <n/p'$, $ \beta<n/q'$. This was extended to the Heisenberg group in \cite{Fractional-Heisenberg-Nonlinear}. In this article, we will also prove an analogue of \eqref{Stein-Weiss-Linear} for the bilinear operator $B_{\lambda}$ on the Heisenberg group $\mathbb{H}^n$, see Theorem~\ref{Power weights AB}. We end this section with  the following bilinear interpolation theorem for Lorentz spaces which will be very useful for our purpose.
\begin{theorem}[\cite{Janson-interpolation}, Theorem~3 \cite{Kenig-Stein-MRL} ]  \label{S.  Janson}
Suppose that a bilinear operator $T : L^{p_{i},1} \times L^{q_{i},1} \rightarrow L^{r_{i}, \infty}$, where $0 < p_{i}, q_{i} \leq \infty$, $0 < r_{i} \leq \infty$, for three points $ \left( \frac{1}{p_{i}}, \frac{1}{q_{i}}   \right)$, $i=1,2,3$ in $\R^{2}$, that are non-collinear.
Suppose, further, that there are $\theta_{0}, \theta_{1}, \theta_{2} \in \R $ with $\theta_{1}, \theta_{2} >0$ so that $\frac{1}{r_{i}} = \theta_{0} + \frac{ \theta_{1} }{ p_{i}} + \frac{ \theta_{2} }{ q_{i}}$, $i=1,2,3$. Then,
\begin{equation*}
    T : L^{p} \times L^{q} \rightarrow L^{r},
\end{equation*}
provided $1 \leq p, q \leq \infty$, $ \frac{1}{p} + \frac{1}{q} \geq \frac{1}{r}  $ and  $ \left( \frac{1}{p}, \frac{1}{q} , \frac{1}{r}  \right)$   lies in the open convex hull of $ \left( \frac{1}{p_{i}}, \frac{1}{q_{i}} , \frac{1}{r_{i}}  \right)$.
\end{theorem}

The article is organized as follows. In the next section we prove Theorem~\ref{Kenig Stein Hn}. Section~\ref{Main} is dedicated to the proof of our main result Theorem~\ref{Power weights C}. Finally, as a consequence of Theorem~\ref{Power weights C}, we conclude this article by Stein-Weiss inequality for the bilinear fractional integral operator $B_\lambda$ on $\mathbb{H}^n$. Throughout this article, we write  $A\lesssim B$ and $ B \gtrsim A$ to abbreviate $A\leq C B$ for some constant $C$ is independent of $A$ and $B$, and $A \simeq B $ means both $A\lesssim B$ and $ A \gtrsim B$. We write the Euclidean convolution of $f$ and $g$ on $\mathbb{R}^{2n+1}$ by $f*_{e}\, g  (x):=\int_{\mathbb{R}^{2n+1}} f(y) g(x-y)\, dy$.  

\section{Proof of Theorem \ref{Kenig Stein Hn} }
We first prove part ($b$) in \Cref{Kenig Stein Hn} when $p=q=1$. This is the key estimate for proving \Cref{Kenig Stein Hn}. Let us introduce the following operators which are pieces of the operator $B_{\lambda}$.
\begin{equation*}
    B(f,g)(x) = \int_{|y| \simeq 1} f(xy^{-1}) g(xy) dy,
\end{equation*}
and
\begin{equation*}
    B_{k}(f,g)(x) = \int_{|y| \simeq 2^{-k}} f(xy^{-1}) g(xy) dy.
\end{equation*}
Our main goal is to establish the end-point weak type boundedness $L^{1}(\H^n) \times L^{1}(\H^n) \rightarrow L^{1/2}(\H^n)$ for the pieces $B_{k}$. We address this as the following lemma.
\begin{lemma}\label{pieces Bk}
The following statements hold:
\begin{enumerate}
    \item [\emph{(\emph{i})}] \hspace*{.7mm} $\lVert  B(f,g)   \rVert_{L^{1/2} (\H^n)} \,  \lesssim  \,  \lVert  f  \rVert_{L^{1} (\H^n)} \, \lVert  g  \rVert_{L^{1} (\H^n)}. $
    \item [\emph{(\emph{ii})}] \hspace*{.7mm} $\lVert  B(f,g)   \rVert_{L^{1} (\H^n)}  \,  \lesssim \, \lVert  f  \rVert_{L^{1} (\H^n)} \,  \lVert  g  \rVert_{L^{1} (\H^n)}. $
    \item [\emph{(\emph{iii})}] \hspace*{.7mm} $\lVert  B_{k}(f,g)   \rVert_{L^{1/2} (\H^n)} \, \lesssim \, 2^{-Qk} \  \lVert  f  \rVert_{L^{1} (\H^n)} \,  \lVert  g  \rVert_{L^{1} (\H^n)}. $
    \item [\emph{(\emph{iv})}] \hspace*{.7mm} $\lVert  B_{k}(f,g)   \rVert_{L^{1} (\H^n)} \lesssim \lVert  f  \rVert_{L^{1} (\H^n)} \, \lVert  g  \rVert_{L^{1} (\H^n)}. $
\end{enumerate}
\end{lemma}

\begin{proof}[Proof of \Cref{pieces Bk}]
The statements ($iii$) and ($iv$) follow from ($i$) and ($ii$), respectively, by scaling: Let $r=2^{-k}$ and   $B_{r}(f,g)(x) = \int_{|y| \simeq r} f(xy^{-1}) g(xy) dy$ then
\begin{equation*}
    \delta_{r} \left[ B( \delta_{r} f, \delta_{r} g)\right] = r^{-Q} \, B_{r}(f,g), \ \ \ Q=2n+2.
\end{equation*}
We assume, without loss of generality, that $f \geq 0, \, g \geq 0$. We begin with proving ($i$). For $a \in \Z^{2n+1}$, let $Q_{a} = a \cdot Q_{0}$, where $Q_{0}= [ 0,1)^{2n+1}$. Then,
\begin{equation}
    \begin{split}\label{L1 cross L1 to L1/2}
        \lVert  B(f,g) \,  \chi_{Q_{a}}  \rVert_{L^{1/2} (\H^n)} & \leq \int_{Q_{a}} B(f,g)(x) dx \leq \int_{ x \in Q_{a}} \int_{  y \,  \in B(0,1) } f(xy^{-1}) \, g(xy) \, dy \, dx \\
        & \stackrel{y \rightarrow y^{-1} \cdot x}{=}  
          \int_{ Q_{a}} \int_{   y \,  \in \,  x \cdot B(0,1) } f(y) \, g(xy^{-1}x) \, dy \, dx \\
          & \leq \int_{ Q_{a}} \int_{   y \,  \in \,  Q_{a} \cdot B(0,1) } f(y) \, g(xy^{-1}x) \, dy \, dx \\
          & =  \int_{   y \,  \in \,  Q_{a} \cdot B(0,1) } f(y) \int_{ x \in Q_{a}} g(xy^{-1}x) \, dx \, dy \\
          &  \stackrel{x \rightarrow x \cdot y}{=}  \int_{   y \,  \in \,  Q_{a} \cdot B(0,1) } f(y) \int_{ x \in Q_{a} \cdot y^{-1} } g(x^{2} y ) \, dx \, dy\\
          & \stackrel{x \rightarrow x/2, \  x \rightarrow  x \cdot y^{-1}}{=} 2^{-2n-1} \, \int_{   y \,  \in \,  Q_{a} \cdot B(0,1) } f(y) \int_{ x \in \left(   Q_{a} \cdot y^{-1} \right)^{2} \cdot y } g(x ) \, dx \, dy\\
          & \leq 2^{-2n-1} \, \int_{   y \,  \in \,  Q_{a}^{*}  } f(y) \, dy  \int_{ x \in Q_{a}^{*}  } g(x ) \, dx,
    \end{split}
\end{equation}
where $Q_{a}^{*}:= a \cdot Q_{0} \cdot B(0,1) \cdot Q_{0}^{-1} . Q_{0} \subset a \cdot \left( [-4,4]^{2n} \times [-16, 16] \right)$.

Observe that  $ Q_{a}$ and $Q_{a}^{*}$ have bounded overlapping and covers whole of $\H^{n}$. Indeed, 
let $(z,t)= (x+ i y, t) \in \H^{n}$. Choose an $a' \in \Z^{2n}$ such that $a' \leq (x,y) < a' + (1, \cdots, 1)$, component wise. Having chosen $a'$, choose integer, say $a_{2n+1}$ such that $t - \frac{1}{2} \Im (a'\cdot \bar{z}) \in [a_{2n+1}, a_{2n+1} +1)$. Then we have an $a=(a', a_{2n+1}) \in \Z^{2n+1}$ such $a^{-1} \cdot (z,t) \in Q_{0}$. So $\H^{n} = \cup_{a \in \Z^{2n+1}} Q_{a}$. For bounded overlapping of $Q_{a}$, let us fix an $a \in \Z^{2n+1}$ and consider $\tilde{a} \in \Z^{2n+1}$ such that
\begin{equation*}
  a \cdot Q_{0} \cap \tilde{a} \cdot Q_{0} \neq \emptyset. 
\end{equation*}
Equivalently, $ Q_{0} \cap a^{-1} \tilde{a} \cdot Q_{0} \neq \emptyset $. Let $(z,t) = a^{-1} \cdot \tilde{a}$ and $(w,s) \in Q_{0}$. Let $\| z \|$ means the Euclidean norm of $z \in \C^{n}$. Then, if $\| z\|>2 \sqrt{n}$, then $(z,t) \cdot (w,s) = (z+w, t + s + \frac{1}{2} \Im (z \cdot \bar{w})) \notin Q_{0}$.  If $\| z\| \leq 2 \sqrt{n}$ but $|t| > 2n+2$, then $|t + s + \frac{1}{2} \Im (z \cdot \bar{w}) | \geq 2(n+1) - (n+1)= n+1$. So, again $(z,t) \cdot (w,s)  \notin Q_{0}$. If $\| z\| \leq 2 \sqrt{n}$ and $|t| \leq 2n+2$, then for fixed $a$, we are, at most, counting the number of lattice points $\tilde{a} \in \Z^{2n+1}$ such that $\tilde{a} \in a \cdot B(0, 4 \sqrt{n} )$ which is, clearly, $\simeq {n}^{Q/2} $. Similarly, we can argue for the sets $Q_{a}^{*}$.

So,
\begin{equation*}
    \begin{split}
        \lVert  B(f,g)  \rVert_{L^{1/2} (\H^n)}^{1/2} & \simeq \sum_{a \in \Z^{2n+1}}  \lVert   B(f,g) \,  \chi_{Q_{a}}  \rVert_{L^{1/2} (\H^n)}^{1/2} \\
        & \lesssim \sum_{a \in \Z^{2n+1}}  \lVert  f \,  \chi_{ Q_{a}^{*} }  \rVert_{L^{1} (\H^n)}^{1/2} \, \lVert  g \,  \chi_{ Q_{a}^{*} }  \rVert_{L^{1} (\H^n)}^{1/2}\\
        & \leq \left(  \sum_{a \in \Z^{2n+1}}  \lVert  f \,  \chi_{ Q_{a}^{*} }  \rVert_{L^{1} (\H^n)} \right)^{1/2} \, \left(  \sum_{a \in \Z^{2n+1}}  \lVert  g \,  \chi_{ Q_{a}^{*} }  \rVert_{L^{1} (\H^n)} \right)^{1/2}\\
        & \simeq \lVert  f   \rVert_{L^{1} (\H^n)} \, \lVert  g   \rVert_{L^{1} (\H^n)},
    \end{split}
\end{equation*}
establishing $(i)$.

Next, ($ii$) follows from using the same set of change of variables and Fubini’s theorem as in (\ref{L1 cross L1 to L1/2}). Thus, completing the proof of \Cref{pieces Bk}.
\end{proof}
Returning to the proof of part ($b$) in \Cref{Kenig Stein Hn} when $p=q=1$, $\frac{1}{r} = 2 - \frac{\lambda}{Q}$.  Let $\lVert  f   \rVert_{L^{1} (\H^n)} = \lVert  g   \rVert_{L^{1} (\H^n)} = 1$. Let us decompose the operator $B_{\lambda}$ as
\begin{equation*}
\begin{split}
    B_{\lambda}(f,g)(x)  &  \simeq  \sum_{ k \in \Z} 2^{k(Q- \lambda)} B_{k}(f,g)(x)\\
    & = \sum_{ k \leq k_{0}} + \sum_{ k \geq k_{0}} =: F_{1} + F_{2},
\end{split}
\end{equation*}
and for $F_{1}$ and $F_{2}$ we have, using ($iii$) and ($iv$) in \Cref{pieces Bk},
\begin{equation*}
    \begin{split}
        \lVert  F_{1}   \rVert_{L^{1} (\H^n)} \leq \sum_{ k \leq k_{0}} \, 2^{k(Q- \lambda)}  \lVert   B_{k}(f,g)   \rVert_{L^{1} (\H^n)} \lesssim \sum_{ k \leq k_{0}} \, 2^{k(Q- \lambda)} \simeq  2^{k_{0}(Q- \lambda)}
    \end{split}
\end{equation*}
and
\begin{equation*}
    \begin{split}
        \lVert  F_{2}   \rVert_{L^{1/2} (\H^n)}^{1/2} \leq \sum_{ k > k_{0}} \, 2^{ \frac{k(Q- \lambda)}{2} }  \lVert   B_{k}(f,g)   \rVert_{L^{1/2} (\H^n)}^{1/2} \lesssim \sum_{ k > k_{0}} \, 2^{ \frac{k(Q- \lambda)}{2} } 2^{- \frac{kQ}{2}} \simeq  2^{ -  \frac{\lambda }{2} k_{0} }.
    \end{split}
\end{equation*}
Then, for all $t>0$,
\begin{equation*}
    \begin{split}
        \left| \left\lbrace B_{\lambda}(f,g) > t \right\rbrace \right| & \leq  \left| \left\lbrace F_{1} > \frac{C t}{2} \right\rbrace \right| \, + \,  \left| \left\lbrace F_{2} > \frac{C t }{2} \right\rbrace \right|\\
        & \lesssim \frac{ \lVert  F_{1}   \rVert_{L^{1} (\H^n)} }{ t } + \frac{ \lVert  F_{2}   \rVert_{L^{1/2} (\H^n)}^{1/2} }{ t^{1/2} }\\
        & \lesssim \frac{ 2^{k_{0}(Q- \lambda)} }{t } \, + \,  \frac{ 2^{ -  \frac{\lambda }{2} k_{0} } }{ t^{1/2} }.
    \end{split}
\end{equation*}
Optimising the right hand side of the above with respect to $k_{0}$, that is, choosing $k_{0}$ such that $ 2^{k_{0}(Q- \lambda)} / t  \, =  \,   2^{ -  \frac{\lambda }{2} k_{0} } /  t^{1/2} $, gives the desired estimate
\begin{equation*}
   \left| \left\lbrace B_{\lambda}(f,g) > t \right\rbrace \right| \lesssim \frac{1}{ t^{r} }, \ \ \  \frac{1}{r} = 2 - \frac{\lambda}{Q},
\end{equation*}
which settles the proof of ($b$), in \Cref{Kenig Stein Hn} when $p=q=1$. To finish part ($b$), observe that, if $g \in L^{\infty} (\H^{n} )$, we have
\begin{equation*}
    B_{\lambda} (f,g)(x) \leq \| g \|_{L^{\infty}(\mathbb{H}^n)} \, \left( f * \frac{1}{|y|^{Q - \lambda}} \right) (x), \ \ \ x \in \H^{n}. 
\end{equation*}
So, from linear fractional integration on $\H^{n}$,
\begin{equation*}
  \begin{split}
        \| \,  B_{\lambda}(f,g) \|_{L^{r, \infty}(\mathbb{H}^n  )} 
        & \leq \, \| g \|_{L^{\infty}(\mathbb{H}^n)}  \left\lVert  f * \frac{1}{|y|^{Q - \lambda}}  \right\rVert_{L^{r, \infty}( \mathbb{H}^n  )}\\
        & \lesssim  \| g \|_{L^{\infty}(\mathbb{H}^n)} \, \| f \|_{L^{1}(\mathbb{H}^n)},
  \end{split}
\end{equation*}
if $\frac{1}{r} = 1 - \frac{\lambda}{Q}$ which is, indeed, the situation when $p=1, q= \infty$. If $g \in L^{q} (\H^{n} )$, $1<q< \infty$, then $(b)$ follows from linear interpolation, by fixing $f \in L^{1} (\H^{n} )$, and using bounds for $B_{\lambda}$ in  $ L^{1} (\H^{n} ) \times L^{1} (\H^{n} ) \rightarrow L^{r, \infty} (\H^{n} )$ (the key estimate) and $ L^{1} (\H^{n} ) \times L^{\infty} (\H^{n} ) \rightarrow L^{r, \infty} (\H^{n} )$ (the sub case discussed above). 

The part $(a)$ is obtained from $(b)$, by applying bilinear interpolation \Cref{S. Janson}. For the sake of completeness, we briefly explain this point from \cite{Kenig-Stein-MRL}. Consider the open convex set in $\R^2$,
\begin{equation*}
    C := \left\lbrace (x,y) \in \R^2 : x + y > \frac{\lambda}{Q}, \, 0<x, y < 1 \right\rbrace, \ \ \ 0< \lambda < Q.
\end{equation*}
Observe that the interior of $C$ is precisely the union of interior of  triangles whose vertices lie on different sides of the square $[0,1]^2$ intersected with $C$. Thus, by symmetry and part ($b$), it suffices to establish the ``weak-type" inequality for $p= \infty$ and $1<q<\infty$. But for $f \in L^{\infty}(\H^{n})$,
\begin{equation*}
    B_{\lambda} (f,g)(x) \leq \| f \|_{L^{\infty}(\mathbb{H}^n)} \, \left( g * \frac{1}{|y|^{Q - \lambda}} \right) (x), \ \ \ x \in \H^{n},
\end{equation*}
which implies
\begin{equation*}
  \begin{split}
        \| \,  B_{\lambda}(f,g) \|_{L^{r}(\mathbb{H}^n  )} 
        & \leq \, \| f \|_{L^{\infty}(\mathbb{H}^n)}  \left\lVert  g * \frac{1}{|y|^{Q - \lambda}}  \right\rVert_{L^{r}( \mathbb{H}^n  )}\\
        & \lesssim  \| f \|_{L^{\infty}(\mathbb{H}^n)} \, \| g \|_{L^{q}(\mathbb{H}^n)},
  \end{split}
\end{equation*}
which is guaranteed by the strong type boundedness of linear fraction operator on $\H^{n}$ provided, $\frac{1}{r} = \frac{1}{q} - \frac{\lambda}{Q}>0$, $1<q<\infty$, which is, indeed, true in this case. .

\section{Characterization of power weights}
\label{Main}

In this section we provide the proof of Theorem~\ref{Power weights C}. Let us start with proving the sufficient part. Our proof of the sufficient part involves delicate analysis of singularities of the operator $B_\alpha$. Subsequently, we decompose it appropriately to estimate each piece individually. In contrast with the proof of \cite{Komori-2020}, we provide a unified approach to handle the operator $B_\lambda$ irrespective of the sign of $\alpha$ and $\beta$.

\subsection{Proof of the sufficient part}\label{Sufficiency}
\begin{proof}[Proof of $\eqref{Necessary-condn-ABC}\implies \eqref{suff-cond-ABC}$:] Since $|B_{\lambda}(f, g)|\leq B_{\lambda}(|f|, |g|)$, throughout this proof we will assume that $f, g$ are non-negative functions. First we will prove the following weak type estimate
\begin{align}\label{weak type estimates}
\|S(f, g)\|_{L^{r, \infty}(\mathbb{H}^n)}\leq K \|f\|_{L^p(\mathbb{H}^n)}\|g\|_{L^q(\mathbb{H}^n)},
\end{align}
where,
\begin{align*}
S(f, g)=S_{\alpha, \beta, \gamma}(f, g)(x):= |x|^{-\gamma}  \int_{\mathbb{H}^n}  \frac{f(x y^{-1})}{|xy^{-1}|^{\alpha}} \, \frac{g(xy)}{|xy|^{\beta}}  \, \frac{dy}{|y|^{Q-\lambda}}.  
\end{align*}
Once the proof of \eqref{weak type estimates} is complete, as an application of bilinear interpolation, we can conclude the required strong type estimates.

We analyse the operator $S$ into three parts. Namely the following: $S(f, g)(x)= \sum_{i=1}^{3}\mathcal{J}_{i}(x)$, where
\begin{align*}
&\mathcal{J}_{1}(x):= |x|^{-\gamma}  \int_{y\in B(0, \frac{|x|}{2 } ) \cdot x}  \frac{f(x y^{-1})}{|xy^{-1}|^{\alpha}} \, \frac{g(xy)}{|xy|^{\beta}}  \, \frac{dy}{|y|^{Q-\lambda}},\\
&\mathcal{J}_{2}(x):=|x|^{-\gamma}  \int_{y\in   x^{-1} \cdot B(0, \frac{|x|}{ 2 }) }  \frac{f(x y^{-1})}{|xy^{-1}|^{\alpha}} \, \frac{g(xy)}{|xy|^{\beta}}  \, \frac{dy}{|y|^{Q-\lambda}},\\
&\mathcal{J}_{3}(x):= |x|^{-\gamma}  \int_{  \H^{n} \setminus \left[  B(0, \frac{|x|}{2 }) \cdot x  \,  \bigcup \,   x^{-1} \cdot B(0, \frac{|x|}{2 }) \right]   }  \frac{f(x y^{-1})}{|xy^{-1}|^{\alpha}} \, \frac{g(xy)}{|xy|^{\beta}}  \, \frac{dy}{|y|^{Q-\lambda}}.
\end{align*}
We first estimate $\mathcal{J}_{3}$.\\

\noindent\underline{\textbf{Estimate for $\mathcal{J}_3$:}}
Let us denote the set $ \H^{n} \setminus \left[  B(0, \frac{|x|}{2 }) \cdot x  \,  \bigcup \,   x^{-1} \cdot B(0, \frac{|x|}{2 }) \right] $ by $G_{x}$. One can decompose $\mathcal{J}_{3}$ as follows:
\begin{align*}
\mathcal{J}_{3}(x)&\leq \mathcal{J}_{31}(x)+\mathcal{J}_{32}(x),\,\,\text{where,}\\
&\mathcal{J}_{31}(x):= |x|^{-\gamma}  \int_{\{y: |y|\geq 2|x|\}}  \frac{f(x y^{-1})}{|xy^{-1}|^{\alpha}} \, \frac{g(xy)}{|xy|^{\beta}}  \, \frac{dy}{|y|^{Q-\lambda}},\\
&\mathcal{J}_{32}(x):= |x|^{-\gamma}  \int_{\{y\in G_{x} : |y|< 2|x|\}}  \frac{f(x y^{-1})}{|xy^{-1}|^{\alpha}} \, \frac{g(xy)}{|xy|^{\beta}}  \, \frac{dy}{|y|^{Q-\lambda}}.
\end{align*}
Observe that for $y$ such that $|y|\geq 2|x|$ we have $|xy^{-1}|\simeq |y|$ and $|xy|\simeq |y|$. According to our hypothesis $-Q+\lambda\leq \alpha+\beta$. First, let us consider the case when $-Q+\lambda=\alpha+\beta$. Therefore,
\begin{align*}
\mathcal{J}_{31}(x)\lesssim |x|^{-\gamma} (f*_{e} g)(2x).   
\end{align*}
We also have $\gamma\geq -\alpha-\beta=Q-\lambda>0$, and $\frac{1}{r}=\frac{\gamma}{Q}+\frac{1}{s}$, where, $\frac{1}{s}=\frac{1}{p}+\frac{1}{q}-1$. Using H\"older's inequality for weak-type spaces and Young's convolution inequality subsequently, we obtain $\|\mathcal{I}_{31}\|_{L^{r, \infty}(\mathbb{H}^{n})}\lesssim \|f\|_{L^p(\mathbb{H}^{n})} \|g\|_{L^q(\mathbb{H}^{n})}$, which is the required estimate.
\medskip

When $-Q+\lambda<\alpha+\beta$, together with the condition $\gamma>Q/r$, we can ensure that $\lambda-\alpha-\beta<Q(\frac{1}{p}+\frac{1}{q})$. Now choosing $\mu>0$ such that $\mu\in (\lambda-\alpha-\beta-\gamma, Q(\frac{1}{p}+\frac{1}{q}))\cap (\lambda-\alpha-\beta, Q)$, we conclude the following
\begin{align}\label{estimate-1}
\nonumber\mathcal{J}_{31}(x) 
& \lesssim |x|^{-\gamma}  \int_{ \{ y: |y| \geq 2|x| \} }  \frac{ f(x y^{-1})} { |y|^{\alpha} } \, \frac{ g(xy) } { |y|^{\beta} }  \, \frac{dy}{ |y|^{Q-\lambda}},  \\ 
\nonumber & =  |x|^{-\gamma}   \int_{ \{ y: |y| \geq 2|x| \} } \frac{ |y|^{ -\alpha-\beta+\lambda-\mu}} { |y|^{ Q-\mu } }  f(x y^{-1}) g(xy) \, dy\\
&\lesssim |x|^{- \gamma - \alpha - \beta + \lambda - \mu}  B_{ \mu }(f, g)(x).
\end{align}
Define $\frac{1}{s}:=\frac{1}{p}+\frac{1}{q}-\frac{\mu}{Q}$, then it is trivial to see that $\frac{1}{r}=\frac{1}{s}+\frac{\alpha+\beta+\gamma+\mu-\lambda}{Q}$. Denote $h(x)=|x|^{-\gamma-\alpha-\beta-\mu+\lambda}$. Using H\"older's inequality for weak-type spaces we obtain
\begin{align*}
\|\mathcal{J}_{31}\|_{L^{r, \infty}(\mathbb{H}^n)}\lesssim \|h\|_{L^{\frac{Q}{\alpha+\beta+\gamma+\mu-\lambda}, \infty}}\|B_{\mu}(f, g)\|_{L^{s, \infty}}\lesssim K_{\alpha,\beta, \gamma, Q, \lambda}\|f\|_{L^p(\mathbb{H}^n)}\|g\|_{L^q(\mathbb{H}^n)},
\end{align*}
where we have used Theorem~\ref{Kenig Stein Hn} in the last inequality. This completes the estimates for $\mathcal{J}_{31}$.
\medskip

To estimate $\mathcal{J}_{32}$, observe that for points $y \in G_{x}$ with $|y|<2|x|$, we have $|xy^{-1}|=|yx^{-1}|\simeq |x|$ and $|xy|\simeq |x|$. Assuming that $\alpha+\beta+\gamma>0$, one can choose $\mu_1\in (0, Q(\frac{1}{p}+\frac{1}{q}))$ such that $-\alpha-\beta-\gamma+\lambda<\mu_{1}<\lambda$. Now
\begin{align*}
\nonumber\mathcal{J}_{32}(x)&\lesssim |x|^{-\gamma-\alpha-\beta} \int_{\{y: |y|<2|x|\}}\frac{|y|^{\lambda-\mu_1}}{|y|^{Q-\mu_1}} f(x y^{-1}) g(xy)\, dy \\
&\lesssim |x|^{-\gamma-\alpha-\beta+\lambda-\mu_1} B_{\mu_1}(f, g)(x).
\end{align*}
At this point we follow the argument provided after equation~\eqref{estimate-1} to conclude that
$\|\mathcal{J}_{32}\|_{L^{r, \infty}(\mathbb{H}^n)}\lesssim \|f\|_{L^{p}(\mathbb{H}^n)}\|g\|_{L^{q}(\mathbb{H}^n)}$. Similarly, if $\alpha+\beta+\gamma=0$, we have $\mathcal{J}_{32}(x)\lesssim B_{\lambda}(f, g)(x)$, then also we have the required estimate invoking Theorem~\ref{Kenig Stein Hn}.  \\
\medskip

\noindent\underline{\textbf{Estimate for $\mathcal{J}_{1}$:}} Let $y\in B(0, \frac{|x|}{2 }) \cdot x$ then $y= \xi \cdot x$ for some $ \xi \in B(0, \frac{|x|}{2 })$. Observe that $x y = x \cdot \xi \cdot x = 2x + \xi $ and $|xy|=|2x + \xi |\geq | 2x| - | \xi |\geq c|x|$, for some fixed constant $1<c<\infty$. Moreover, $|xy|\lesssim |x|$. Again observe that $|y|=| \xi \cdot x|\leq | \xi |+|x|\leq \frac{3}{2} |x|$, $|y|=| \xi \cdot x | \geq |x|-| \xi |\geq \frac{1}{2} |x|$. Incorporating these estimates we obtain
\begin{align}
\label{equation-need}
\mathcal{J}_{1}(x)\lesssim |x|^{-\gamma-\beta-Q+\lambda}\int_{y\in  B(0, \frac{|x|}{2 }) \cdot x} \frac{f(x y^{-1})g(xy)}{|xy^{-1}|^{\alpha}} \, dy. 
\end{align}
Our hypothesis $\alpha<Q/p'$ and $-Q+\lambda< \beta+\gamma$ allow us to choose $\mu>0$ such that $Q(1-\frac{1}{p}-\frac{1}{q})<\mu<Q(1-\frac{1}{p})$ and $\alpha< \mu< \alpha+\beta+\gamma+Q-\lambda$. Now
\begin{align*}
\mathcal{J}_{1}(x)&\lesssim |x|^{-\gamma-\beta-Q+\lambda} \int_{y\in  B(0, \frac{|x|}{2 }) \cdot x} |x y^{-1}|^{\mu-\alpha}\frac{f(x y^{-1})g(xy)}{|xy^{-1}|^{\mu}} \, dy\\
& \lesssim |x|^{-\gamma-\beta-Q+\lambda+\mu-\alpha}  \int_{y\in  B(0, \frac{|x|}{2 }) \cdot x} \frac{f(x y^{-1})g(xy)}{|xy^{-1}|^{\mu}} \, dy~~(\text{since}~|xy^{-1}|\leq c|x|)\\
&\stackrel{z=x. y^{-1}}{\lesssim}\,\, |x|^{-\gamma-\beta-Q+\lambda+\mu-\alpha} \int_{\mathbb{R}^{2n+1}}\frac{f(z)g(2x+z^{-1})}{|z|^{\mu}}\ dz\\
&\lesssim |x|^{-\gamma-\beta-Q+\lambda+\mu-\alpha} \left(\frac{f}{|\cdot|^{\mu}}*_{e}g\right)(2x).
\end{align*}
Write $\frac{1}{s}=\frac{1}{p}+\frac{1}{q}-1+\frac{\mu}{Q}$, then $\frac{1}{r}=\frac{1}{s}+\frac{\alpha+\beta+\gamma+Q-\mu-\lambda}{Q}$. This implies 
\begin{equation*}
\|\mathcal{J}_{1}\|_{L^{r, \infty}(\mathbb{H}^n)} \lesssim  \||x|^{-\gamma-\beta-Q+\lambda+\mu-\alpha}\|_{ L^{\frac{Q}{\alpha+\beta+\gamma+Q-\lambda-\mu}, \infty}(\mathbb{R}^{2n+1})} \Big\|\left(\frac{f}{|\cdot|^{\mu}} *_{e} g \right)\Big\|_{L^{s, \infty}(\mathbb{C}^n\times \mathbb{R})}. 
\end{equation*}

Define $\frac{1}{t}=\frac{1}{p}+\frac{\mu}{Q}$, then $\frac{1}{s}=\frac{1}{t}+\frac{1}{q}-1$. By Young's inequality, we obtain $$\Big\|\left(\frac{f}{|\cdot|^{\mu}}*_{e}g\right)\Big\|_{L^{s, \infty}(\mathbb{R}^{2n+1})}   \lesssim \Big\|\frac{f}{|\cdot|^{\mu}}\Big\|_{L^{t, \infty}(\mathbb{R}^{2n+1})} \|g\|_{L^{q}(\mathbb{R}^{2n+1})}.$$ The required estimate i.e., $\|\mathcal{J}_{1}\|_{L^{r, \infty}(\mathbb{H}^n)}\lesssim \|f\|_{L^p(\mathbb{H}^n)}\|g\|_{L^q(\mathbb{H}^n)}$, follows once we use the inequality $\Big\|\frac{f}{|\cdot|^{\mu}}\Big\|_{L^{t, \infty}(\mathbb{R}^{2n+1})}\lesssim \|f\|_{L^p(\mathbb{H}^n)}$. This completes this case. 
\medskip

We are left with the case when $\beta+\gamma=-Q+\lambda$. As a consequence of the condition $\alpha+\beta+\gamma\geq 0$ we obtain $\alpha\geq Q-\lambda>0$. Now \eqref{equation-need} implies
\begin{align*}
\mathcal{J}_{1}(x)\lesssim \int_{y\in  B(0, \frac{|x|}{2 }) \cdot x} \frac{f(x y^{-1})g(xy)}{|xy^{-1}|^{\alpha}} \, dy\lesssim \left(\frac{f}{|\cdot|^{\alpha}}*_{e}g\right)(2x).
\end{align*}
Observe that in this case $\frac{1}{r}=\frac{1}{p}+\frac{1}{q}-1+\frac{\alpha}{Q}$. Define $\frac{1}{t}=\frac{1}{p}+\frac{\alpha}{Q}$, therefore $\frac{1}{r}=\frac{1}{t}+\frac{1}{q}-1$. Therefore, by Young's inequality, we obtain
\begin{align*}\|\mathcal{J}_{1}\|_{L^{r, \infty}(\mathbb{R}^{2n+1})}&\lesssim \Big\|\left(\frac{f}{|\cdot|^{\alpha}} *_{e} g \right)\Big\|_{L^{r, \infty}(\mathbb{R}^{2n+1})}\\
&\lesssim \Big\|\frac{f}{|\cdot|^{\alpha}}\Big\|_{L^{t, \infty}(\mathbb{R}^{2n+1})}\|g\|_{L^{q}(\mathbb{R}^{2n+1})}\lesssim  \|f\|_{L^p(\mathbb{H}^n)} \|g\|_{L^q(\mathbb{H}^n)}.
\end{align*}
\medskip
\noindent \underline{\textbf{Estimate for $\mathcal{J}_2$:}} This case is similar to $\mathcal{J}_1$, so we skip it.

This completes the proof of inequality~\ref{weak type estimates}. Once we have the week type inequalities, achieving the strong type inequality just uses the multiplinear interpolation Theorem~\ref{S.  Janson}. We explain it here. For fixed $\alpha, \beta$ and $\gamma$ in \Cref{Power weights C}, we have
\begin{equation}\label{range of p q}
0< \frac{1}{p} < \min \left[ 1, \ 1 - \frac{\alpha}{Q} \right], \ \ \   \text{and}   \ \ \ 0< \frac{1}{q} < \min \left[ 1, \ 1 - \frac{\beta}{Q} \right].
\end{equation}
The condition $ \frac{1}{r} \leq \frac{1}{p} +\frac{1}{q} $ is equivalent to $\alpha + \beta + \gamma\leq \lambda$, which combined with $0 \leq \alpha + \beta + \gamma$, gives $0 \leq \alpha + \beta + \gamma \leq \lambda$. Further, the conditions $r< \infty$ and $\gamma < \frac{Q}{r}$ being, respectively, equivalent to $ \frac{  \lambda - (\alpha + \beta + \gamma)  }{ Q  } < \frac{1}{p} + \frac{1}{q}$ and $ \frac{  - \alpha - \beta + \lambda  }{ Q  } < \frac{1}{p} + \frac{1}{q}$,
lead to
\begin{equation*}
\frac{  - \alpha - \beta + \lambda  }{ Q  } +      \max  \left[ -    \frac{  \gamma  }{ Q  }, \ 0  \right] < \frac{1}{p} + \frac{1}{q}.
\end{equation*}

Therefore, consider the open convex set in $\R^2$,

\begin{equation*}
    \begin{split}
 C_{\alpha, \beta, \gamma} :=  \left\lbrace   (x,y) \in \left(0,1\right)^2 \ : \  x + y >          \frac{  - \alpha - \beta + \lambda  }{ Q  }  \right.  & +   \max  \left[ -    \frac{  \gamma  }{ Q  }, \ 0  \right],    \\
          & \left.
      \hspace*{-9em}    0<x < \min \left[ 1, \ 1 - \frac{\alpha}{Q}\right],  \    0 < y < \min \left[ 1, \ 1 - \frac{\beta}{Q} \right] \,   \right\rbrace.  
    \end{split}
\end{equation*}
Depending on the sign of $\alpha, \beta$ and $\gamma$, the set $C_{\alpha, \beta, \gamma} \subseteq (0,1)^2$ changes. But, in all cases, for each point $(\frac{1}{p}, \frac{1}{q})$ in $C_{\alpha, \beta, \gamma}$ one can always choose  three non-collinear points inside $C_{\alpha, \beta, \gamma}$  such that
$(\frac{1}{p}, \frac{1}{q})$ is contained in the interior of the solid triangle inside $C_{\alpha, \beta, \gamma}$, determined by these three points.
 Therefore, in view of \Cref{S.  Janson}, it suffices to show the ``weak-type" inequality for $(\frac{1}{p}, \frac{1}{q})$ in $C_{\alpha, \beta, \gamma}$.

\end{proof}

\subsection{The necessary conditions}\label{Necessity}
In \cite{Komori-2020}, some counter examples were constructed to conclude necessary conditions for the boundedness of $BI_{\lambda}$ on the real line. Here, we construct them on the Heisenberg group $\H^{n}$ of any  dimension.

Recall that the inequality \eqref{suff-cond-ABC} is equivalent to the following unweighted boundedness 
\begin{align}
\label{equivalent-criterion}
\|S_{ \alpha, \beta, \gamma}(f, g)\|_{L^r(\mathbb{H}^n)}\leq K \|f\|_{L^p(\mathbb{H}^n)}\|g\|_{L^q(\mathbb{H}^n)}, 
\end{align}
where the operator $S_{\alpha, \beta, \gamma}(f, g)$ is defined as follows
\begin{align}
S_{\alpha, \beta, \gamma}(f, g)(x):= |x|^{-\gamma}  \int_{\mathbb{H}^n}  \frac{f(x y^{-1})}{|xy^{-1}|^{\alpha}} \, \frac{g(xy)}{|xy|^{\beta}}  \, \frac{dy}{|y|^{Q-\lambda}}.  
\end{align}

\noindent\underline{ \textbf{Necessity of $-Q + \lambda \leq \alpha + \beta$ in \Cref{Power weights C} }}: Suppose $\alpha + \beta < - Q + \lambda$. Since $\gamma < \frac{Q}{r}$, whence
\begin{equation}\label{assumption 0}
    (-\alpha - \beta -Q + \lambda), \ \     \frac{Q}{r} - \gamma >0.
\end{equation}
Also, recall that
\begin{equation}\label{def of r}
    \frac{1}{r}= \frac{1}{p} + \frac{1}{q}- \frac{\lambda}{Q}  + \frac{ \alpha + \beta + \gamma}{Q},
\end{equation}
which implies that $\frac{1}{p} + \frac{1}{q} > 1$.
\medskip

Let $N,M \gg 1$ to be specified later. For $a \in \Z^{2n+1}$, consider sets
\begin{equation*}
    E_{a} = a \cdot \delta_{r_{a}} Q(0,1) \cdot a,
\end{equation*}
where $r_{a}= |a|^{-N-1}$ and $Q(0,r):= [0,r]^{2n} \times [0,r^{2}]$, $0<r< \infty$. Observe that 
\begin{equation*}
    E_{a} = 2a + Q(0,r_{a}).
\end{equation*}
Here, $``+"$ denotes usual addition in $\R^{2n+1}$. Take functions
\begin{equation*}
     f( \xi) := \sum_{a \in \Z^{2n+1} \setminus \{0\} } |a|^{M/p} \, \chi_{E_a}( \xi ), \ \ \xi \in \H^n,
 \end{equation*}
 and
\begin{equation*}
     g( \xi) := \sum_{a \in \Z^{2n+1} \setminus \{0\} } |a|^{M/q} \, \chi_{E_a}( \xi ), \ \ \xi \in \H^n,
 \end{equation*}
The functions $f \in L^p(\H^n)$ and $ g \in L^q(\H^n)$ if 
\begin{equation}\label{assumption 1}
    M < Q N.
\end{equation}
We will show that for these choice of functions, $ S_{\alpha ,\beta , \gamma} (f,g) (x) \, \chi_{|x| \ll 1} \notin L^{r} (\H^n) $.
\medskip

Fix $x \in \R_{+}^{2n+1}:= (0, \infty)^{2n+1}$ and choose $K_0 \gg 1$ such that $(K_0 +1)^{-N-1} \leq |x| < K_0^{-N-1}$. Fix $a$ such that $|a| < K_{0}/2$. Consider sets
\begin{equation*}
      \widetilde{ E }_{a,x}  :=  E_{a} \cdot x \, \cap  \, x^{-1} \cdot  E_{a}.
\end{equation*}
By definition, whenever $y \in \widetilde{ E }_{a,x}$ then $y x^{-1}$, $xy \in E_{a}$. 

For $|a| < K_{0}/2$, we see that $| \widetilde{ E }_{a,x} | \gtrsim r_{a}^{Q} = |a|^{-QN - Q}$. Indeed, we observe that $$\widetilde{ E }_{a,x} = x^{-1} \cdot \left( x \cdot E_{a} \cdot x \, \cap  \,   E_{a} \right) = x^{-1} \cdot \left(  \left[ 2x + 2a + Q(0,r) \right] \bigcap  \left[ 2a + Q(0,r) \right] \right).$$ Thus, $| \widetilde{ E }_{a,x} | = \left[ \prod _{j=1}^{2n} (r_{a} - 2x_{j}) \right] $ $(r_{a}^{2} - 2 x_{2n+1} ).$ From our choice of $K_{0}$, we have $x_{2n+1} < K_{0}^{-2(N+ 1 )}$, whence $r_{a}^{2} - 2 \, x_{2n+1} > |a|^{-2(N+1)}$ $ - 2 K_{0}^{-2(N+ 1 )} \gtrsim |a|^{-2(N+1)} = r_{a}^{2}$, since $|a| < K_{0}/2$. Similarly, $r_{a}- 2 x_{j} \gtrsim r_{a}$, $j = 1, \dots, 2n$. Combining the above estimates, we have $| \widetilde{ E }_{a,x} | \gtrsim r_{a}^{Q}$.

 For $y \in \widetilde{ E }_{a,x}$, we have $|xy|$, $|y x^{-1}|$,  $|y| \simeq |a|$. Indeed, $xy \in E_{a}$ so $xy = a \cdot \xi \cdot a = 2a + \xi$ for some $\xi \in Q(0,r)$. Thus, $|xy|^{4} = \left[ \sum_{j=1}^{2n} (2a_{j} + \xi_{j}  )^{2}  \right]^{2} + (2a_{2n+1} + \xi_{2n + 1}  )^{2} \geq |a|^{4}$. So, we have $|a| \leq |xy| \leq |x| + |y|$ which gives $|y|\geq |a| - |x| \gtrsim |a|$. Since $y \in E_{a} \cdot x$, so $|y| \lesssim |a|$ which in turn implies $|xy|, |y| \simeq |a|$. Similarly, $|y x^{-1}| \simeq |a|$.
 
 For fixed $|x| \ll 1$, the collection $\{\widetilde{ E }_{a,x}\}_{a\in \mathbb{Z}^{2n+1}}$ is a disjoint family of sets. Indeed, $\widetilde{ E }_{a,x}$'s are disjoint if and only if the sets $x \cdot E_{a} \cdot x \, \cap  \,   E_{a} = \left[ 2x + 2a + Q(0,r) \right] \bigcap  \left[ 2a + Q(0,r) \right]$ are disjoint, which is true since $|x| \ll 1$.
 
 Therefore, for $(K_0 +1)^{-N-1} \leq |x| < K_0^{-N-1}$,
\begin{equation*}
    \begin{split}
    S_{\alpha ,\beta , \gamma} (f,g) (x) & \gtrsim
         \ |x|^{ - \gamma }  \sum_{ a \in \Z^{2n+1} \setminus \{ 0 \}  : |a| < K_{0}/2 } \int_{y \in \widetilde{ E }_{a,x}}  \frac{f(y x^{-1})}{|y x^{-1}|^{\alpha}} \, \frac{g(xy)}{|x y|^{\beta}}      \frac{dy}{|y|^{Q- \lambda}}\\
        & \gtrsim |x|^{ - \gamma } \sum_{ a \in \Z^{2n+1} \setminus \{ 0 \}  : |a| < K_{0}/2 } |a|^{     (-\alpha - \beta-Q+ \lambda) + M \left( \frac{1}{p} + \frac{1}{q} \right)    }  |a|^{-QN -Q}.
    \end{split}
\end{equation*}
Assuming,
\begin{equation}\label{assumption 2}
    M \left( \frac{1}{p} + \frac{1}{q} \right) -QN >0,
\end{equation}
we are dealing with the sum of the form
\begin{equation*}
    \sum_{ a=(a', a_{2n+1}) \in \Z^{2n+1} \setminus \{ 0 \}  : \left( |a'|^{4} + a_{2n+1}^{2} \right)^{1/4} \leq K_{0}/2 }    |a|^{R - Q} \simeq K_{0}^{R}, \ \ \ R>0.
\end{equation*}
Hence, 
\begin{equation*}
   S_{\alpha ,\beta , \gamma} (f,g) (x) \gtrsim \  |x|^{      - \gamma  -  \frac{  (- \alpha - \beta-Q+ \lambda) +  M \left( \frac{1}{p} + \frac{1}{q} \right) -QN         }{  N+1 }     } \  \chi_{|x| \ll 1},
\end{equation*}
which implies $\| S_{\alpha ,\beta , \gamma} (f,g) \|_{L^{r}(\H^n)}$ will diverge if 
\begin{equation*}
    \begin{split}
      & \gamma  +  \frac{  (- \alpha - \beta -Q+ \lambda) +  M \left( \frac{1}{p} + \frac{1}{q} \right) -QN         }{  N+1 } \geq \frac{Q}{r}\\
      & \iff  \gamma (N+1)  -Q(N+1) + \lambda -( \alpha + \beta ) + M \left( \frac{1}{p} + \frac{1}{q} \right) \\
      & \hspace*{7em} \geq \ \frac{Q}{r} + \frac{QN}{r} =   \frac{Q}{r} +  QN \left( \frac{1}{p} + \frac{1}{q} \right) - N \lambda + (\alpha + \beta + \gamma) N \\
      & \iff  (- \alpha - \beta - Q + \lambda) (N+1) > \left(  \frac{Q}{r} - \gamma  \right)  +  (QN - M) \left( \frac{1}{p} + \frac{1}{q} \right).
    \end{split}
\end{equation*}
Here, we have used (\ref{def of r}). First pick out $N \gg 1$.  Since $\frac{1}{p} + \frac{1}{q} > 1$, we can choose $M$ close to $QN$ such that (\ref{assumption 1}) and (\ref{assumption 2}) are satisfied. Subsequently, the last inequality holds true for large $N$ because of (\ref{assumption 0}). Thus we arrive at a contradiction. Therefore, we must have $\alpha+\beta\geq -Q+\lambda$. 

\medskip

\noindent\underline{ \textbf{Necessity of $ -Q + \lambda \leq  \beta + \gamma$ in \Cref{Power weights C} }:}  Assume $- \beta - \gamma - Q + \lambda > 0$. For $x \in \R_{+}^{2n + 1}$ such that $|x| \gg 1$, consider the following portion of $S_{\alpha, \beta, \gamma}   (f,g)(x)$:
\begin{equation}\label{y close to x}
   \int_{y \, \in Q(0, |x|/ 2\sqrt{n}) \cdot x}  \frac{|x|^{ - \gamma }}{|xy|^{\beta} \, |y|^{Q- \lambda}}  f(x y^{-1}) g(xy)\,  \frac{dy}{|xy^{-1}|^{ \alpha }}. 
\end{equation}
Arguing as in the previous example, we see that if $y \in Q(0, |x|/2 \sqrt{n}) \cdot x$ then $|xy|, \  |y| \simeq |x|$.

Therefore, (\ref{y close to x}) is bounded below by a constant times of the following
\begin{equation}\label{y close to x 1}
 |x|^{ - \beta - \gamma  + \lambda -Q}  \int_{y \, \in Q(0, |x|/ 2\sqrt{n}) \cdot x}    f(x y^{-1}) g(xy)\,  \frac{dy}{|xy^{-1}|^{ \alpha  }}. 
\end{equation}
Take $f(y)=  |y|^{-s}  \, \chi_{Q(0,1)}(y)$ with $ s < \frac{Q}{p}$ so that $f \in L^p( \H^n)$. Performing the change of variables $y \rightarrow y \cdot x$,  (\ref{y close to x 1}) becomes 
\begin{equation}\label{y close to x 2}
\begin{split}
     |x|^{   - \beta - \gamma  + \lambda -Q  }  \int_{y \, \in Q(0, |x|/ 2\sqrt{n}) }     g(x \cdot y \cdot x)\,  \frac{dy}{|y|^{ \alpha + s}}   .
\end{split} 
\end{equation}
Next, we choose
 \begin{equation*}
     g( \xi) := \sum_{a \in \Z^{2n+1} : |a|>e} |a|^{Q(N-1)/q} \, \left( \log |a| \right)^{-2/q} \chi_{E_a}( \xi ), \ \ \xi \in \H^n,
 \end{equation*}
where $E_a:= a \cdot Q(0,r_a) \cdot a$, $r_a= \frac{1}{|a|^{N}}$. The function $g \in L^q(\H^n)$, which follows from disjointness of the sets $E_a= 2a + Q(0,r_a)$  and their measure $|E_{a}|=|Q(0,r_a)|= |a|^{-QN}$.

Setting $\Tilde{E_{a}}:= a^{1/2} \cdot Q(0,r_a /4) \cdot a^{1/2}$, where the notation $(z, t)^{1/2}$ means $ ( z/2, t/2)$ for $(z,t) \in \H^{n}$.

 If $x \in \Tilde{E_{a}}:= a^{1/2} \cdot Q(0,r_a /4) \cdot a^{1/2}$ and $y \, \in Q(0, r_a /2) $, then $xyx \in E_{a}= a \cdot Q(0,r_a) \cdot a$ and $|x| \simeq |a|$. Indeed, for such $x$ and $y$,
 $xyx = 2x + y \in 2 \left( a^{1/2} \cdot Q(0,r_a /4) \cdot a^{1/2} \right) +  Q(0, r_a /2)$ $= 2 \left[ a +  Q(0,r_a /4)  \right] + Q(0, r_a /2) \subset 2a + 2 Q(0, r_a/4) + Q(0,r_a/2) \subset 2a + Q(0, r_a) = a \cdot Q(0, r_a) \cdot a =E_{a} $. Also, $x = a + \xi$, for some $\xi = (\xi', \xi_{2n+1}) \in [0, \frac{r_a}{4}]^{2n} \times [0, (\frac{r_a}{4})^{2} ] $. Writing $a = (a', a_{2n+1})$, we have $|x|^{4} = \|a' + \xi' \|^{4} + (a_{2n+1} + \xi_{2n+1})^{2} \geq |a|^{4}$, where  $\| a' \|$ is the Euclidean norm of $a' \in \R^{2n+1}$. 
 
 Further, since $\Tilde{E_{a}}= a^{1/2} \cdot Q(0,r_a /4) \cdot a^{1/2} = a + Q(0,r_a /4)$, so clearly $|\Tilde{E_{a}}| = |Q(0,r_a)| \simeq r_a^{Q} = |a|^{-QN}$ and $\{\Tilde{E_{a}}\}_{a\in \mathbb{Z}^{2n+1}}$ is a disjoint collection.
 
Incorporating the above, \eqref{y close to x 2} implies
\begin{equation*}
    \begin{split}
       \|S_{ \alpha, \beta, \gamma}(f, g)\|_{L^r(\mathbb{H}^n)}^{r} & \gtrsim  \sum_{|a|>e} \,  \int_{x \in \Tilde{E_a}} |x|^{(  - \beta - \gamma  + \lambda -Q  )r } \left| \int_{y \, \in Q(0, r_{a}/2) }     g(x \cdot y \cdot x)\,  \frac{dy}{|y|^{ \alpha + s}}    \right|^r dx\\
        & \simeq \sum_{|a|>e} \,   |a|^{(  - \beta - \gamma  + \lambda -Q  )r} |a|^{    \frac{ rQ(N-1) }{q}  }  \left( \log |a|\right)^{- \frac{2r}{q} } |a|^{(s+ \alpha -Q) N r} |a|^{-QN}.
    \end{split}
\end{equation*}
Therefore, $\|S_{ \alpha, \beta, \gamma}(f, g)\|_{L^r(\mathbb{H}^n)}$ diverges provided

\begin{equation*}
\begin{split}
         (  - \beta & - \gamma  + \lambda -Q  )r + \frac{ rQ(N-1) }{q} + (s + \alpha -Q) N r - QN \geq -Q \\
        & \iff  - \beta - \gamma  + \lambda -Q + \frac{ Q(N-1) }{q} + (s + \alpha -Q) N \geq  \frac{ Q(N-1) }{r}\\
        & \iff  - \beta - \gamma  + \lambda -Q + \frac{ Q(N-1) }{q} + ( s + \alpha -Q ) N \\
        & \hspace*{10em} \geq    \left( \frac{ QN }{p} - \frac{1}{p} \right) +    \frac{ Q(N-1) }{q} +  (N-1) ( \alpha + \beta + \gamma - \lambda)\\
        & \iff  N \left( ( -\beta - \gamma   - Q + \lambda  ) - \left( \frac{Q}{p} -s  \right) \right) \geq      \frac{1}{p'}  - \alpha. 
\end{split}
\end{equation*}
 Since $ -\beta - \gamma   - Q + \lambda >0$, we choose $s < \frac{Q}{p}$ sufficiently close to $\frac{Q}{p}$ so that $( -\beta - \gamma   - Q + \lambda  ) - \left( \frac{Q}{p} -s  \right) >0$  and then, taking $N$ large, we have that the last inequality holds true.

\medskip

\noindent\underline{\textbf{Necessity of $ \alpha + \beta + \gamma \geq 0$ in \Cref{Power weights C}}}: Contrarily,  suppose $\gamma_{0}:= \alpha + \beta + \gamma <0$. Then, the homogeneity condition takes the form of  $\frac{1}{r}= \frac{1}{p} + \frac{1}{q}- \frac{\lambda - \gamma_{0} }{Q} >0 $.

Consider the portion of $S_{\alpha, \beta, \gamma}   (f,g)(x)$ in  the set $\{y \in \H^n : |y| \ll |x| \}$, wherein, one has $|x| \gtrsim |xy|, |x y^{-1}| \geq | \, |x|-|y| \, | \gtrsim |x|$:
\begin{equation}\label{y smaller than x}
 S_{\alpha, \beta, \gamma}   (f,g)(x) \,  \gtrsim \, |x|^{-\gamma_{0} }  \int_{\{y\in \H^n |y| \ll |x|\}}  f(x y^{-1}) g(xy)\,  \frac{dy}{|y|^{Q-\lambda}}   . 
\end{equation}
Take $N \gg 1$, to be specified later, and consider functions
\begin{equation*}
    f(y)= \sum_{a \in \Z^{2n+1} : |a|>e} |a|^{Q(N-1)/p} \, \left( \log |a| \right)^{-2/p} \chi_{E_a}(y), 
\end{equation*}
and 
\begin{equation*}
    g(y)= \sum_{a \in \Z^{2n+1} : |a|>e} |a|^{Q(N-1)/q} \, \left( \log |a| \right)^{-2/q} \chi_{E_a}(y), \ \ y \in \H^n,
\end{equation*}
with $E_a:= a \cdot B(0,r_a)^{2}$, $r_a= \frac{1}{|a|^{N}}$. Observe that $|E_a|=|B(0,r_{a})^2| \sim r_a^{Q} = \frac{1}{|a|^{NQ}  }$ and that $E_{a}, a \in \Z^{2n+1}$,  are disjoint sets. Therefore, $f \in L^p(\H^n)$ and $g \in L^q(\H^n)$.

Define  sets $\Tilde{E_a}:= a \cdot B(0,r_a)$. For $x \in \Tilde{E_a}$ and $y \in B(0,r_a)$, we have $x y^{-1}, \, xy \in a \cdot B(0, r_a)^2 = E_a$.
Therefore, from  (\ref{y smaller than x}),
\begin{equation*}
\begin{split}
   \|S_{ \alpha, \beta, \gamma}(f, g)\|_{L^r(\mathbb{H}^n)}^{r} 
& \gtrsim \sum_{|a|>e} \,  \int_{x \in \Tilde{E_a}}  \left| |x|^{-\gamma_{0} }  \int_{|y| < r_{a}} f(x y^{-1}) g(xy)\,  \frac{dy}{|y|^{Q-\lambda}}    \right|^r dx\\
& \gtrsim \sum_{|a|>e} \,  \int_{x \in \Tilde{E_a}} |x|^{-\gamma_{0} \, r} |a|^{ rQ(N-1) (\frac{1}{p} + \frac{1}{q})     }  \left( \log |a|\right)^{-2r (\frac{1}{p}+ \frac{1}{q} )}  \frac{1}{|a|^{N \lambda r}} dx\\
& \simeq \sum_{|a|>e} \,   |a|^{-\gamma_{0} \, r +  rQ(N-1) (\frac{1}{p} + \frac{1}{q})  - r N \lambda   }  \left( \log |a|\right)^{-2r (\frac{1}{p}+ \frac{1}{q} )}  \frac{1}{|a|^{N Q}} \\
& =: \sum_{|a|>e} \,   |a|^{R}  \left( \log |a|\right)^{-2r (\frac{1}{p}+ \frac{1}{q} )},
\end{split}
\end{equation*}
which diverges if $R \geq -Q$, wherein we have used that the sets  $\Tilde{E_a}, \, a \in \Z^{2n+1}$, are disjoint, and that  for $x \in \Tilde{E_a} = a \cdot B(0, r_{a})$,  $|x| \sim |a|$. Therefore, in view of the homogeneity condition, it suffices to check whether
\begin{equation*}
\begin{split}
 &   -\gamma_{0} \, r +  rQ(N-1) \left(\frac{1}{p} + \frac{1}{q} \right)  - r N \lambda  - NQ \geq -Q\\
 & \iff  -\gamma_{0}  +  Q(N-1) \left(\frac{1}{p} + \frac{1}{q} \right)  - N \lambda   \geq \frac{(N-1)Q}{r} \\
 & \iff  -\gamma_{0}  -  Q \left(\frac{1}{p} + \frac{1}{q} \right)  + QN \left(\frac{1}{p} + \frac{1}{q} - \frac{\lambda}{Q} - \frac{1}{r} \right) +  \frac{Q}{r}  \geq    0 \\
 & \iff  -\gamma_{0}  -  Q \left(\frac{1}{p} + \frac{1}{q} \right)  - \gamma_{0} \,  N +  \frac{Q}{r}  \geq    0,
 \end{split}
\end{equation*}
which is true for $N$ sufficiently large, since $\gamma_{0} < 0$.

\noindent\underline{ \textbf{Necessity of $\frac{1}{r} \leq \frac{1}{p} + \frac{1}{q}$ in \Cref{Power weights C} }}:\\
\noindent On the contrary, let us assume $\frac{1}{r}>\frac{1}{p}+\frac{1}{q}$. Take $f(x):=|x|^{-Q/p} (\log |x|)^{-\tau_{1}}\chi_{\{|x|>16\}}$ and $g(x):=|x|^{-Q/q} (\log |x|)^{-\tau_{2}}\chi_{\{|x|>16\}}$. It is not hard to see that $f\in L^p(\mathbb{H}^n)$ and  $g\in L^q(\mathbb{H}^n)$ provided $\tau_{1}>1/p$ and $\tau_{2}>1/q$, respectively. For  $|x|\gg 1$, we see that
\begin{align*}
S_{\alpha, \beta, \gamma}(f, g)(x)&\gtrsim |x|^{- \gamma}    \int_{B(0, \frac{|x|}{2})}  \frac{f(x y^{-1})}{|xy^{-1}|^{\alpha}} \, \frac{g(xy)}{|xy|^{\beta}}  \, \frac{dy}{|y|^{Q-\lambda}}\\
&\gtrsim |x|^{-(\alpha+\beta+\gamma)}\int_{B(0, \frac{|x|}{2})}  f(x y^{-1}) g(xy) \, \frac{dy}{|y|^{Q-\lambda}}\\
&\gtrsim |x|^{-(\alpha+\beta+\gamma)-Q(\frac{1}{p}+\frac{1}{q})+\lambda} (\log |x|)^{\tau_1+\tau_2}=|x|^{-Q/r}(\log |x|)^{\tau_1+\tau_2}.
\end{align*}
The above implies that $\|S_{ \alpha, \beta, \gamma}(f, g)\|_{L^r(\mathbb{H}^n)}=\infty$ if we choose $\tau_1 > 1/p$ and $\tau_2 > 1/q$ such that $(\tau_1+\tau_2)\leq \frac{1}{r}$ and which is possible thanks to the assumption $\frac{1}{r}>\frac{1}{p}+\frac{1}{q}$. For example one can choose $\tau_1=\frac{1}{p}+\epsilon$ and $\tau_2=\frac{1}{q}+\epsilon$ with $\epsilon>0$ such that $0<2\epsilon<\frac{1}{r}-\frac{1}{p}-\frac{1}{q}$. Thus we arrive at at a contradiction.  
\medskip

\subsection{Bilinear Stein--Weiss Inequality}

As a consequence of Theorem~\ref{Power weights C}, we obtain the following bilinear Stein--Weiss inequality.
\begin{theorem}[Stein-Weiss type inequality]\label{Power weights AB}
Let $1<p,q<\infty$, $0< \lambda <Q$, $0<r<\infty$, and $\frac{1}{r}= \frac{1}{p} + \frac{1}{q}- \frac{\lambda}{Q} >0$. Further, let 
\begin{equation}\label{Easy AB}
\alpha < \frac{Q}{p'}, \ \ \ \beta < \frac{Q}{q'} \ \ \  \text{and}   \ \ \ - \frac{Q}{r} < \alpha + \beta.
\end{equation}
Then, the following are equivalent:
\begin{enumerate}
\item [\emph{(\emph{a})}] There exists constant $K>0$, such that
\begin{equation}
\label{Suuf-Cond-AB}
\| \, |x|^{ \alpha + \beta } B_{\lambda}(f,g) \|_{L^{r}(\mathbb{H}^n  )} \leq K  \| |x|^{ \alpha } f \|_{L^p( \mathbb{H}^n  )} \| |x|^{ \beta }g \|_{L^q( \mathbb{H}^n  )},
\end{equation}
for all $f \in L^{p}(\mathbb{H}^n)$ and $g \in L^{q}(\mathbb{H}^n)$;
\item [\emph{(\emph{b})}] The exponents $\alpha$ and $\beta$ satisfy
\begin{equation}\label{Necessary condn. AB}
\textnormal{(I)} \  \alpha \leq Q- \lambda, \ \ \  \textnormal{(II)} \  \beta \leq Q - \lambda, \ \ \  and  \ \ \ \textnormal{(III)} \  -Q + \lambda \leq \alpha + \beta.
\end{equation}
\end{enumerate}
\end{theorem}

\section*{Acknowledgments} 
We are sincerely thankful to Prof. S. Thangavelu for suggesting this direction.

\newcommand{\etalchar}[1]{$^{#1}$}
\providecommand{\bysame}{\leavevmode\hbox to3em{\hrulefill}\thinspace}
\providecommand{\MR}{\relax\ifhmode\unskip\space\fi MR }
\providecommand{\MRhref}[2]{%
  \href{http://www.ams.org/mathscinet-getitem?mr=#1}{#2}
}
\providecommand{\href}[2]{#2}


\begin{thebibliography}{LOP{\etalchar{+}}09}

\bibitem[BLO21]{Brascamp-Lieb-PAMS}
R.~M. Brown, C.~W. Lee, and K.~A. Ott, \emph{Estimates for {B}rascamp-{L}ieb
  forms in {$L^p$}-spaces with power weights}, Proc. Amer. Math. Soc.
  \textbf{149} (2021), no.~2, 747--760.

\bibitem[CF11]{Fractional-Lie-Group}
Jiecheng Chen and Dashan Fan, \emph{A bilinear fractional integral on compact
  {L}ie groups}, Canad. Math. Bull. \textbf{54} (2011), no.~2, 207--216.

\bibitem[FS74]{Folland-Stein}
G.~B. Folland and E.~M. Stein, \emph{Estimates for the {$\bar \partial _{b}$}
  complex and analysis on the {H}eisenberg group}, Comm. Pure Appl. Math.
  \textbf{27} (1974), 429--522.

\bibitem[Gra92]{Grafakos-1992}
Loukas Grafakos, \emph{On multilinear fractional integrals}, Studia Math.
  \textbf{102} (1992), no.~1, 49--56.

\bibitem[HLZ12]{Fractional-Heisenberg-Nonlinear}
Xiaolong Han, Guozhen Lu, and Jiuyi Zhu, \emph{Hardy-{L}ittlewood-{S}obolev and
  {S}tein-{W}eiss inequalities and integral systems on the {H}eisenberg group},
  Nonlinear Anal. \textbf{75} (2012), no.~11, 4296--4314.

\bibitem[HM18]{Moen-Hoang}
Cong Hoang and Kabe Moen, \emph{Weighted estimates for bilinear fractional
  integral operators and their commutators}, Indiana Univ. Math. J. \textbf{67}
  (2018), no.~1, 397--428.

\bibitem[Jan88]{Janson-interpolation}
Svante Janson, \emph{On interpolation of multilinear operators}, Function
  spaces and applications ({L}und, 1986), Lecture Notes in Math., vol. 1302,
  Springer, Berlin, 1988, pp.~290--302.

\bibitem[Kai14]{Kairema-homogeneous}
Anna Kairema, \emph{Sharp weighted bounds for fractional integral operators in
  a space of homogeneous type}, Math. Scand. \textbf{114} (2014), no.~2,
  226--253.

\bibitem[KF20]{Komori-2020}
Yasuo Komori-Furuya, \emph{Weighted estimates for bilinear fractional integral
  operators: a necessary and sufficient condition for power weights}, Collect.
  Math. \textbf{71} (2020), no.~1, 25--37.

\bibitem[KS99]{Kenig-Stein-MRL}
Carlos~E. Kenig and Elias~M. Stein, \emph{Multilinear estimates and fractional
  integration}, Math. Res. Lett. \textbf{6} (1999), no.~1, 1--15.

\bibitem[LOP{\etalchar{+}}09]{Lerner-Adv-Multi-Weight}
Andrei~K. Lerner, Sheldy Ombrosi, Carlos P\'{e}rez, Rodolfo~H. Torres, and
  Rodrigo Trujillo-Gonz\'{a}lez, \emph{New maximal functions and multiple
  weights for the multilinear {C}alder\'{o}n-{Z}ygmund theory}, Adv. Math.
  \textbf{220} (2009), no.~4, 1222--1264.

\bibitem[LT97]{Lacey-Thiele}
Michael Lacey and Christoph Thiele, \emph{{$L^p$} estimates on the bilinear
  {H}ilbert transform for {$2<p<\infty$}}, Ann. of Math. (2) \textbf{146}
  (1997), no.~3, 693--724.

\bibitem[Moe14]{Kabe-TAMS}
Kabe Moen, \emph{New weighted estimates for bilinear fractional integral
  operators}, Trans. Amer. Math. Soc. \textbf{366} (2014), no.~2, 627--646.

\bibitem[MW74]{MW-fractional}
B.~Muckenhoupt and R.~Wheeden, \emph{Weighted norm inequalities for fractional
  integrals}, Trans. Amer. Math. Soc. \textbf{192} (1974), 261--274.

\bibitem[SW58]{Stein-Weiss}
E.~M. Stein and Guido Weiss, \emph{Fractional integrals on {$n$}-dimensional
  {E}uclidean space}, J. Math. Mech. \textbf{7} (1958), 503--514.

\end{thebibliography}
\end{document}